\documentclass[a4paper]{article}

\usepackage{amsfonts}
\usepackage{latexsym,amssymb}
\usepackage{amsmath,amsthm}

\newtheorem{theorem}{Theorem}[section]
\newtheorem{lem}[theorem]{Lemma}
\newtheorem{prop}[theorem]{Proposition}
\newtheorem{cor}[theorem]{Corollary}

\newtheorem{defin}[theorem]{Definition}

\def\vp{\varphi}

\begin{document}

\title{On  logically-geometric types  of algebras }

\maketitle

\begin{center}

\author{G.Zhitomirski}
     
 \smallskip
        {\small
                Department of Mathematics,
          Bar-Ilan University,

          52900, Ramat Gan, Israel

                {\it E-mail address:}zhitomg@012.net.il
        }

\end{center}

\begin{abstract}
The connection between classical model theoretical types (MT-types) and logically-geometrical types (LG-types) introduced by B. Plotkin is considered. It is proved that MT-types of two $n$-tuples in two universal algebras  coincide  if and only if their  LG-types coincide. An algebra $H$  is called logically perfect if for every two $n$-tuples in $H$ whose types coincide, one can be sent to another by means of an automorphism of this algebra. Some sufficient condition for logically perfectness of  free finitely generated algebras is given which helps to prove that finitely generated free Abelian groups, finitely generated free nilpotent groups and finitely generated free semigroups are logically perfect. It is proved that if two Abelian groups have the same type and one of them is finitely generated and free then these groups are isomorphic.

\end{abstract}



\section{Introduction}\label{S_Int}
The ideas suggested and developed by  B. Plotkin in the field of algebraic logic  seem to be very interesting and efficient.  It turns  out that the geometrical notions and the  geometrical intuition  can be successfully applied  in studying  algebras from arbitrary varieties.  Such approach leads to so called universal algebraic geometry and multi-sorted logical geometry. 

The sketch of the ideas of universal algebraic geometry, problems and  results  can be found for example in \cite{Plotkin_AG}, \cite{Plotkin_7Lec}, \cite{Plotkin_SomeResultsUAG}, \cite{PlotZitom_1}, \cite{PlotZitom_2}. The notions  of logical geometry and obtained facts are presented in  \cite{PZ},\cite{Plotkin_IsotAlg}, \cite{Plotkin}.

The purpose of this paper is to consider only one but important notion of model theory, namely, the notion of type. The model theoretic notion of type is well known \cite{MT}. Such a type is denoted in the paper  by MT-type. MT-type is related to one-sorted logics. On the other hand, the ideas of universal logical geometry  give rise to logically-geometric types (LG-types).  This notion is related to multi-sorted logic \cite{PZ},\cite{Plotkin}. Some of the problems discussed in the literature are the following ones: how are connected algebraically two $n$-tuples in an algebra whose types coincide and what we can say about two algebras whose types coincide. 

Let $\Theta $ be  a variety of universal algebras of some signature and $W(X)$ denote the  free $\Theta$-algebra  
over a set $X=\{x_1, x_2, \dots , x_n\}$. In the universal  algebraic geometry, the set  $A^n$ of all  $n$-tuples in a $\Theta $-algebra $A$ is replaced by the set $\hom (W(X), A)$ which is called an $n$-dimensional affine space and whose elements are called points. Since a point $\mu \in \hom (W(X), A)$ is a map we can speak  about its kernel. Along with this usual kernel, so called logical kernel  of $\mu$ is defined. The notion of logical kernel of a point leads to the notion of  LG-type of an algebra.  All notions mentioned above are defined in Section \ref{S_Prel}.

Although the two kinds of types mentioned above are related to different languages we show that MT-types of two $n$-tuples coincide if and only if the logical kernels of the corresponding points coincide (Theorem \ref{twoTypes}). 

Then in Section \ref{LPA} we consider so called logically perfect algebras. An algebra $H$ is said to be logically perfect if for
 every its  two $n$-tuples  whose types coincide there exists  an automorphism of $H$ which  sends one of these tuples to another.  A sufficient condition for logically perfectness of  free finitely generated algebras is given. The main result in this section is Theorem \ref{three log.perf}. 

The last Section \ref{Iso} is devoted to  algebras having the same type (isotyped algebras).   It is proved (Theorem \ref{IsoAbel}) that if two Abelian groups have the same type and one of them is finitely generated and free then these groups are isomorphic.

The obtained results solve some problems set in \cite{Plotkin}.

{\bf Acknowledgments}
The author is pleased to thank B. Plotkin for useful discussions and interesting suggestions.  

\section{Preliminaries}\label{S_Prel}
Throughout this paper, $\Theta $ is a variety of universal algebras of some signature which determines the corresponding first-order language $L$ with equality "$\equiv $" and the infinite set $X^0 =\{x_1, x_2, \dots \}$ of variables. Let $W(X)$ denote the free $\Theta $-algebra generated by $X\subset  X^0$. We consider finite subsets $X\subset  X^0$  only and  follow the conception suggested by B. Plotkin (see for example  \cite{Plotkin_AG}, \cite{Plotkin}, \cite{Plotkin_IsotAlg}, \cite{PZ}).

 Let $\mathbb M$ be a $\Theta $-algebra with the domain $M$. Every  $n$-tuple $\bar{a}=(a_1,\dots ,a_n)$ of elements of $M$ determines a homomorphism $\mu : W(X)\to \mathbb M$ where $X=\{x_1, \dots , x_n \}$, viz $\mu (x_i)=a_i$ for $i=1,\dots, n$. And vice versa, every such homomorphism determines an $n$-tuple in $M$. Thus the set $M^n$ can be identified with  $ \hom (W(X),\mathbb M )$ which is called an affine space and whose elements are called points. Considering the tuples in $M$ as  points in the corresponding  affine space gives us new interesting opportunities. 

First of all, the kernel of a point $\mu$ appears: $Ker \mu =\{(w,w')\vert \mu (w)=\mu (w')\}$.  It is useful to consider equalities $w\equiv w'$ instead of corresponding pairs in $W(X)$. Such an approach leads to connections  between sets of points and systems of identities, that is, to something like to algebraic geometry for an universal algebra. For details see papers cited above. In the present paper, we focus  on the notion of so called logical kernel of a point $\mu$: $LKer \mu$. We recall the definition according to \cite{Plotkin}. 

 Let $\Gamma$ denote the set of all finite subsets of $X^0$. For every $X\in \Gamma$, consider the signature $L_X= \{\vee, \wedge, \neg, \exists x, x \in X, M_X, \} $, where  $M_X$ is the  set of all equalities $w\equiv w'$, where $w,w'\in W(X)$. By adding for every  $X\in \Gamma$ symbols  $s=s^{XY}:W(X)\to W(Y)$, we obtain multi-sorted signature $L_\Theta$ . 
The corresponding multi-sorted language is defined by induction on length and sort of formulas. 

\begin{defin}\label{formula}
\item 1. Each equality $w\equiv w'$ is a formula of the length zero and sort  $X$ if $w\equiv w'\in M_{X}$.
\item 2. Let $u$ be a formula of the length $n$ and the sort $X$. Then the formulas $\neg u$ and $\exists x u$ are the formulas of the same sort $X$ and  the length $(n+1)$. 
\item 3. For the given $s:W(X)\to W(Y)$ we have the formula $s_{*}u$ with the length $(n+1)$ and the sort $Y$.
\item 4. Let  $u_1$ and $u_2$ be formulas of the same sort $X$ and the length $n_1$ and $n_2$ accordingly. Then the formulas $(u_1\vee u_2)$ and $(u_1\wedge u_2)$ have the length $(n_1+n_2+1)$ and the sort $X$.  

 The set of all formulas of the sort $X$ will be denote by $\Phi (X)$.
\end{defin}
The value $Val^{X}_{H}(u)$ of a formula  $u \in \Phi (X)$ in a $\Theta $-algebra $H$ is defined according to the construction. Elements of $Val^{X}_{H}(u)$ are points $\mu : W(X) \to H$.

\begin{defin}\label{Val}

\item  (1).  $Val^{X}_{H}(w\equiv w')=\{ \mu \mid \mu (w)=\mu(w') \}$.

\item (2). If $v=\exists x u$ and  $u\in \Phi (X)$, then $\mu \in Val^{X}_{H}(v)$ if and only if there exists a point $\nu :W(X)\to H$ such that $\nu$ coincides with $\mu$ for all $y\in X$ besides $x$ and $\nu \in Val^{X}_{H}(u)$. 

\item (3). If $u_1, u_2 \in \Phi (X)$  then $Val^{X}_{H}(u_1\vee u_2)=Val^{X}_{H}(u_1) \cup Val^{X}_{H}(u_2)$, $Val^{X}_{H}(u_1\wedge u_2)=Val^{X}_{H}(u_1) \cap Val^{X}_{H}(u_2)$.

\item (4). $Val^{X}_{H}(\neg u)=\hom (W(X),H )\setminus Val^{X}_{H}(u)$.

\item (5). Let $s:W(X)\to W(Y)$ be a homomorphism,  $v\in \Phi (X)$ and $u=s_{*}v$. Then $\mu \in Val^{Y}_{H} (u)$ if and only if $\mu \circ s \in Val^{X}_{H} (v)$.

\end{defin}
\begin{defin}\label{LKer}
A formula $u\in \Phi(X)$ belongs to the logical kernel $LKer(\mu)$
of a point $\mu: W(X)\to H$ if and only if $\mu\in Val^{X}_{H}(u)$.

The set $LKer(\mu)$ of formulas from $\Phi(X)$ is  called {\it logically-geometric $X$-type} of the point $\mu$ ($X$-$LG$-type). 

\end{defin}
\begin{defin}[\cite{PZ}]
The set $T$ of formulas from $\Phi(X)$ is called  $X$-$LG$-type of the algebra $H$, if there is a point $\mu : W(X)\to H$
such that $T=LKer(\mu)$. 

Algebras $H_1$ and $H_2$ in $\Theta$ are called $LG$-isotyped, if for any
finite $X$, every $X$-type of the algebra $H_1$ is an $X$-type of the
algebra $H_2$ and vice versa.
\end{defin}

\begin{defin}
An algebra $H$ is called logically perfect if for every two points $\mu$ and $\nu$  in $H$ having the same $X$-$LG$-type (that is, $LKer(\mu)=LKer(\nu)$) there exists an automorphism $\varphi$ of $H$ such that $\mu =\varphi \circ \nu$, that is, $\varphi$ transports $n$-tuple $(\nu (x_1), \dots ,\nu(x_n))$ to $n$-tuple $(\mu (x_1), \dots ,\mu(x_n))$.
\end{defin}

Now we recall the model-theoretical notion of  type of an $n$-tuple  $\bar{a}$. 
\begin{defin}\label{MT_Type}
The type $tp^{\mathbb M}(\bar{a})$ consists of all formulas $u(x_1,\dots ,x_n) \in L$ with free variables  $x_1,\dots ,x_n$ (all other variables in this formula are bounded) such that $\mathbb M \models u(a_1, \dots ,a_n)$, that is, $u(x_1,\dots ,x_n)$ is true under interpretation which assigns $a_i$ to $x_i$.
\end{defin}
Such a type will be called $MT$-type.  It is worth to mention that we do not consider types depending of parameters (the more general definition can be found in \cite{chloe}). The problem arises  how two tuples are algebraically connected  if their $MT$-types coincide.  

Two kinds of types defined above (MT- and LG-type) are  sets of formulas in different languages. We will prove below  that the two points  $\mu$ and  $\nu$ have the same $X$-$LG$-type if and only if the  $n$-tuples  $(\nu (x_1), \dots ,\nu(x_n))$ and  $(\mu (x_1), \dots ,\mu(x_n))$ have the same $MT$-type. 


\section{Relations between  logical-geometrical types and model-theoretical types}\label{S_Types}

\begin{theorem}\label{twoTypes}
Let $H_1$ and $H_2$ be $\Theta $-algebras. Let  $\bar{a}=(a_1,\dots ,a_n)$ and $\bar{b}=(b_1,\dots ,b_n)$ be $n$-tuples in $H_1$ and $H_2$ respectively. Consider two corresponding points $\nu :W(X)\to H_1$ and  $\mu :W(X)\to H_2$ where $X=\{x_1, \dots , x_n \}$, $\nu(x_i)=a_i$ and $\mu(x_i)=b_i$, $i=1,\dots ,n$. Then $LKer(\nu)=LKer(\mu)$ if and only if $tp^{H_1}(\bar{a})=tp^{H_2}(\bar{b})$.
\end{theorem}
\begin{proof}
We will prove this statement by several steps.
\begin{lem}\label{1}
 $LKer(\nu)=LKer(\mu) \Rightarrow tp^{H_1}(\bar{a})=tp^{H_2}(\bar{b})$.
\end{lem}
\begin{proof} 

Let $LKer(\nu)$=$LKer(\mu)$. Let $u\in tp^{H_1}(\bar{a})$. Under Definition \ref{MT_Type}, we have that $u= u(x_1,\dots ,x_n,\; y_1,\dots ,y_m)$ with listed variables, where $x_1, \dots , x_n$ have free occurrences only, all $y_1,\dots ,y_m$ are bounded, and $H_1 \models u(a_1, \dots ,a_n)$. 

On the other hand, according to  Definition \ref{formula}, $u\in  \Phi (X\cup Y)$, where $Y=\{y_1,\dots ,y_m \}$.  Therefore for every homomorphism  $\gamma :W(X\cup Y) \to H_1$ such that $\gamma (x_i)=a_i$,$i=1,\dots ,n$ (values of $\gamma (y_j)$ do not  influence), we have $\gamma \in Val^{X\cup Y}_{H_1}(u)$ (see definitions (1)-(4) from \ref{Val}). Consider arbitrary homomorphism $s:W(X\cup Y) \to W(X)$ such that $s(x_i)=x_i$, $i=1,\dots ,n$ and construct the formula $v=s_{*}u \in \Phi (X)$. Since $\nu \circ s (x_i)=a_i$, $i=1,\dots ,n$, we have  $ \nu \circ s \in Val^{X\cup Y}_{H_1}(u)$. Under definition \ref{Val} (5), we obtain that  $\nu \in Val^{X}_{H_1}(v)$ and therefore $v\in LKer \nu$.

Since $LKer(\nu)$=$LKer(\mu)$,  we have $v\in LKer \mu$, that is, $\mu \in Val^{X}_{H_2}(v)$ which implies that $ \mu \circ s \in Val^{X\cup Y}_{H_2}(u)$. Let $\delta :W(X\cup Y) \to H_2$ be an an arbitrary homomorphism such that $\delta (x_i)=b_i$ for all $i=1,\dots ,n$. Since $ \mu \circ s: W(X\cup Y) \to H_2$, $\mu \circ s (x_i)=b_i$ for all $i=1,\dots ,n$ and the variables from $Y$ are bounded in $u$ , we obtain that the values of the formula $u$ under interpretations $\delta $ and $\mu \circ s$ coincide.  
Therefore $H_2 \models u(b_1, \dots ,b_n)$, that is, $u\in tp^{H_2}(\bar{b})$. Consequently  $tp^{H_1}(\bar{a}) \subseteq tp^{H_2}(\bar{b})$. The inverse inclusion is also true by symmetry.  
\end{proof}

Now we assign to every formula $u\in \Phi (X),\; X\in \Gamma,$  a  formula $\tilde{u}$ in the one-sorted first order language, that is, a formula which does not contain symbols $s_*$. Let $\tilde{X}^0$  be a copy of $X^0$ such that  to every variable $x\in X^0$  the variable $\tilde{x}\in \tilde{X}^0$ corresponds one to one. Consider the first-order language $L$ associated with the variety $\Theta$ with set $X^0 \cup \tilde{X}^0$ of variables using variables from $X^0$ for free variables and variables from $\tilde{X}^0$ for bounded ones only.

We  construct the formula $\tilde{u}$ for every formula  $u\in \Phi (X), \; X\in \Gamma$,  inductively.

1. If $u$ is $w\equiv w'$ then $\tilde{u}=u$.

2. If $u$ is $\neg v, (u_1\vee u_2)$ or $(u_1\wedge u_2)$ then $\tilde{u}=\neg \tilde{v}, \; (\tilde{u_1}\vee \tilde{u_2})$ or $(\tilde{u_1}\wedge \tilde{u_2})$ respectively.

3. If $u=\exists x v$ and $x\in X$ then $\tilde{u}=\exists \tilde{x} \tilde{v}|^x _{\tilde{x}}$, where $\tilde{v}|^x _{\tilde{x}}$ denotes the formula in $L$ which is obtained by replacing of all occurrences of the variable $x$ in $\tilde{v}$  by  $\tilde{x}$.

4.  Let $Y=\{y_1,\dots ,y_m\}\in \Gamma$ and $s:W(Y)\to W(X)$ be a homomorphism,  $v\in \Phi (Y)$ and $u=s_{*}v$. Then $\tilde{u}=\tilde{v}|^{y_1} _{s(y_1)},\dots ,^{y_m} _{s(y_m)}$. Notice that all occurrences of elements from $X$ and $Y$ in $\tilde{v}$ can be free only.
\begin{lem}
 For every point $\mu :W(X) \to H$ and every $u\in \Phi (X)$
$$u\in LKer(\mu) \Leftrightarrow \tilde{u}\in tp^{H}(\bar{a}) ,$$
where $ \bar{a} =(\mu (x_1), \dots ,\mu(x_n))$, $X=\{x_1,\dots ,x_n\}$.
\end{lem}\label{AssociatedForm}
\begin{proof}
We will prove this statement by induction according to the construction of formulas of sort $X$.

1. Let $u$ be $w\equiv w'$. Under definition, $u\in LKer(\mu)$ means that $\mu (w)=\mu (w')$. In the considered case,  $\tilde{u} =u$ and we obtain that $u\in LKer(\mu)$ is equal to   $H \models \tilde{u}(a_1, \dots ,a_n)$, that is, to  $\tilde{u} \in tp^{H}(\bar{a})$.

2. For  $u=\neg v, (u_1\vee u_2$) or $(u_1\wedge u_2)$ our statement is obviously true.

3. Let $u=\exists x v$ , where  $x\in X$.  Assume that our statement is true  for $v$. The fact $u\in LKer(\mu)$ means that there exists a point $\nu :W(X)\to H$ which coincides with $\mu$ for all $y\in X$ besides $x$ and such that $\nu \in Val^{X}_{H}(v)$. Under assumption, $\nu \in Val^{X}_{H}(v)$ is equal   to $\tilde{v} \in tp^{H}(\bar{b})$ where $\bar{b}=(\nu (x_1),\dots ,\nu (x_n))$. Since   $\tilde{u}=\exists \tilde{x} \tilde{v}|^x _{\tilde{x}}$,  we obtain that $u\in LKer(\mu)$ is equal to $\tilde{u} \in tp^{H}(\bar{a})$ where $\bar{a}=(\mu (x_1),\dots ,\mu (x_n))$. Notice that $\tilde{u}$ does not contain $x$.

4. Let $Y=\{y_1,\dots ,y_m\}$, $s:W(Y)\to W(X)$ be a homomorphism,   $v\in \Phi (Y)$,  and $u=s_{*}v$.  Assume that our statement is true for $v$. This means that $v\in LKer (\mu \circ s)$ is equal to $\tilde{v} \in tp^{H}(\bar{b})$, where  $\bar{b}=\mu \circ s(\bar{y})$.  Further, $v\in LKer (\mu \circ s)$ is equal to $u\in LKer (\mu) $ and $\tilde{v} \in tp^{H}(\bar{b})$ is equal to $\tilde{u} \in tp^{H}(\bar{a})$ because $\tilde{u}=\tilde{v}|^{y_1} _{s(y_1)},\dots ,^{y_m} _{s(y_m)}$ according to the definition, and hence $H\models \tilde{u}(a_1,\dots ,a_n)$ is the same that $H\models \tilde{v}(b_1,\dots ,b_m)$. Thus our statement is true for $u$ too.

\end{proof}
\begin{lem}\label{2}
$ tp^{H_1}(\bar{a})=tp^{H_2}(\bar{b})\Rightarrow LKer(\nu)=LKer(\mu) $
\end{lem}
\begin{proof}
Let $tp^{H_1}(\bar{a})=tp^{H_2}(\bar{b})$. Let $u\in \Phi(X)$ and $u\in LKer \nu$. Then according to  Lemma \ref{AssociatedForm}, $\tilde{u}\in tp^{H_1}(\bar{a})$. Consequently  $\tilde{u}\in tp^{H_2}(\bar{b})$ and therefore $u\in LKer \mu$  according to the same Lemma.
\end{proof}
In virtue of Lemmas \ref{1} and \ref{2}, Theorem \ref{twoTypes} is proved.
\end{proof}

\section{Logically perfect algebras}\label{LPA}
The purpose of this section is to present some results concerning logically perfect algebras. Some authors call an algebra $H$ {\it  homogeneous} if every automorphism between two finitely generated subalgebras of $H$ can be extended to an automorphism of $H$. It is easy to see that every homogeneous algebra is logically perfect \cite{Plotkin_IsotAlg}. 

It is obvious that every finite dimensional linear space $V$  is a homogeneous algebra, and therefore $V$  is logically perfect. On the other hand, it is easy to see that  free finitely generated semigroups and free finitely generated Abelian groups are not homogeneous, nevertheless we will show below that all of them are  logically perfect. Thus the homogeneity  is not a necessary condition for an algebra to be logically perfect. There is a logical condition equivalent to homogeneity obtained by the author. This condition  is cited in \cite{Plotkin_IsotAlg} and called there {\it strictly logically perfectness}.  The following generalization of homogeneity will be useful.

\begin{defin}
An algebra $H$ is called {\it weakly homogeneous} if for every isomorphism $\varphi :A\to B$ between two its finitely generated subalgebras $A$ and $B$,  the following condition is satisfied: if  $\varphi$ itself and its inverse map $\varphi ^{-1} :B\to A$ both can be extended to endomorphisms of $H$ then $\varphi$ can be extended to an automorphism of  $H$.
\end{defin}
\begin{theorem}\label{W_Hom}
Every weakly homogeneous finitely generated free algebra is logically perfect.
\end{theorem}
\begin{proof}
Let $H$ be weakly homogeneous and $e_1,...,e_n$ be free generators of $H$.
Let $X=\{x_1 ,...,x_k \} $. Consider two points $\nu ,\mu : W(X)\to H$ and suppose that $LKer\nu=LKer\mu$. Let $\nu (x_i )=a_i$ and $\mu (x_i) =b_i$ for all $i=1,...,k$. Take $Y=\{y_1,...,y_n\}$, such that $X\bigcap Y=\emptyset$, and define a  homomorphism $\gamma : W(Y)\to H$   by the values: $\gamma (y_i)= e_i, \; i=1,...,n$. Let  $w_1,...,w_k\in W(Y)$  be any $k$  words such that $a_i=\gamma (w_i), \;i=1,..., k$.

Consider a formula $u$ of sort $X$ of the kind $u=s_* (v)$ where 
$$
v=(\exists y_1)...(\exists y_n)(x_1\equiv w_1\wedge ...\wedge x_k\equiv w_k)
$$
and $s: W(X\bigcup Y)\to W(X)$ defined by   $s(x_i )=x_i ,s(y_1)=...=s(y_n)=x_1$.

It is obvious that $\nu \in Val_H (u)$. Thus under assumption, $\mu \in Val_H (u)$ and therefore $\mu\circ s \in Val_H (v)$. The last one means that there exists a homomorphism $\delta : W(Y)\to H$ such that $b_i=\delta (w_i), \;i=1,..., k$. Define an endomorphism $\sigma$ of $H$ setting $\sigma (e_i)=\delta (y_i), \; i=1,..., n$, that is, $ \sigma \circ \gamma =\delta$. We have  $\sigma (a_i)=\sigma (\gamma (w_i))= \delta (w_i)=b_i$ for $i=1,...,k$. Hence  $\sigma$ determines a homomorphism $\varphi$ of the subalgebra $A$ generated by $a_1,...,a_k$ on the subalgebra $B$ generated by $b_1,...,b_k$. 

Similarly, we can define an endomorphism $\tau$ of $H$ such that $\tau (b_i)=a_i$ for $i=1,...k$. Consequently $\sigma\circ  \tau (b_i)=b_i$ and $\tau \circ  \sigma (a_i)=a_i$ which means that the restriction $\varphi$ of $\sigma$ to $A$ is an isomorphism  of $A$ on $B$ and  $\varphi ^{-1}$ is a restriction of $\tau$. Since $H$ is weakly homogeneous, $\varphi$ can be extended up to automorphism $\tilde {\varphi}$ of $H$ for which we have $\tilde {\varphi}\circ \nu =\mu $.
\end{proof}
\begin{lem}\label{abel}
Finitely generated free  Abelian groups and finitely generated free nilpotent groups are weakly homogeneous.
\end{lem}
\begin{proof}
{\bf 1.} We start with considering  Abelian groups. Let $G$ and $F$ be free Abelian groups of the same rank $n$. Let $A$ and $B$ be two subgroups of $G$ and $F$ respectively which are isomorphic by means of an isomorphism $\varphi :A\to B$. We will prove that if $\varphi $ and  $\varphi ^{-1}$ both can be extended up to homomorphisms $\sigma :G\to F$ and $\tau : F\to G$ respectively,  then $\varphi $ can be extended up to an isomorphism of $G$ onto $F$.

It is known (\cite{Magnus}, Theorem 3.5) that there exists a base $g_1,...,g_n$ of $G$ and a base $a_1,...,a_k$ of $A$ such that $a_i =p_i g_i$ for $1\leq i \leq k$ where $p_1, ..., p_k$ are integers and every $p_{i+1}$ is divisible by $p_i$ for $1\leq i \leq k-1$ . Exactly in the same way, there exists a base $f_1,...,f_n$ of $F$ and a base $b_1,...,b_k$ of $B$ such that $b_i =q_i f_i$ for $1\leq i \leq k$ and every integer  $q_{i+1}$ is divisible by the integer $q_i$ for $1\leq i \leq k-1$.

Let  $\sigma (g_i) =\sum _{j=1}^{n} s^j_i f_j$  and $\tau (f_i) =\sum _{j=1} ^{n} t^j_i g_j$. We obtain two integer matrices of order $n$: $S=||s^j_i||$ and $T=||t^j_i||$. Since  $\varphi :A\to B$ is an isomorphism, $\varphi$ provides an invertible integer matrix $||a^j_i||$ of order $k$, where $\varphi (a_i) =\sum _{j=1}^{k} a^j_i b_j$. Let  $||b^j_i||$ be its inverse matrix: $\varphi ^{-1}(b_i) =\sum _{j=1}^{k} b^j_i a_j$.

Since $\sigma (a_i) =\varphi (a_i)$, we obtain $p_i\sigma (g_i)=\sum _{j=1}^{k} a^j_i b_j =\sum _{j=1}^{k} a^j_i q_j f_j$ for $1\leq i \leq k$. Thus for all $1\leq i \leq k$ we have $p_i\sum _{j=1}^{n} s^j_i f_j=\sum _{j=1}^{k} a^j_i q_j f_j$. This  implies that $p_is^j_i = a^j_i q_j $ for $1\leq i,j \leq k$ and $p_i s^j_i =0$ for $1\leq i \leq k$, $k+1\leq j \leq n$. In view of  the definitions of $p_i, q_i$, we have $p_1=q_1$ and $s^j_i =0$ for all $1\leq i \leq k$ and $k+1\leq j \leq n$.

By duality, we obtain  $q_it^j_i = b^j_i p_j $ for $1\leq i,j \leq k$ and $t^j_i =0$ for all $1\leq i \leq k$ and $k+1\leq j \leq n$. Therefore we obtain for all $1\leq i,j \leq k$ :

\begin{equation}
\sum _{l=1}^{k}s^j_l t^l_i=\sum _{l=1}^{k}\frac {a^j_l q_j}{p_l }\frac {b^l_i p_l}{q_i} =\sum _{l=1}^{k}\frac {q_j}{q_i}a^j_l{b^l_i}=
\begin{cases}

1,  &\text{if $i=j$;}\\
0,   &\text{if $i\not= j$.}
\end{cases}
\end{equation}

Consider the left corner $k$-th minor $M$ of the matrix $S$, that is, the determinant of the matrix $||s^j _i||_{1\leq i,j \leq k}$. According to (1) $M=1$ or $M=-1$. Define map $\tilde{\varphi} : G\to F$ setting
\begin{equation}
\tilde{\varphi}(g_i)=
\begin{cases}
\sigma (g_i), &\text{if $i\leq k$;}\\
f_i, &\text{if $ k+1 \leq i\leq n$.}
\end{cases}
\end{equation}
The matrix $V$ of this map is
\begin{equation*}
V =
\left  (
\begin{matrix}
s^1_1 & ...&s^1_k &0&...&0 \\
... & ... & ...&0&...&0\\
s^k_1 & ...&s^k_k &0&...&0 \\
0 & ...&0 &1&...&0 \\
0 & ...&0&0&1...&0 \\
0 &...&0&0&...&1
\end{matrix}
\right )
\end{equation*}

We see that $Det V =M=\pm 1$ and therefore $\tilde{\varphi}$ is an isomorphism. By construction, $\tilde{\varphi}(a_i)=m_i \tilde{\varphi}(g_i)=m_i\sigma(g_i)=\sigma(a_i)=\varphi (a_i)$ for all $i\leq k$. Consequently  $\tilde{\varphi}$ extends $\varphi$.
\par
{\bf 2.} Now let $H$ be a finitely generated free nilpotent group of class $c>1$ and rank $n$. Let $A$ and $B$ be two subgroups of $H$  which are isomorphic by means of an isomorphism $\vp :A\to B$. Let $\vp $ and  $\vp ^{-1}$ both can be extended up to endomorphisms $\sigma $ and $\tau$ of $H$ respectively. 

The quotient group $G=H/H'$ is a free Abelian group of the same rank $n$. Let $\eta : H\to G$ be the corresponding epimorphism. Then $\bar{A} =\eta (A)$ and $\bar{B}=\eta (B)$ are isomorphic subgroups of $G$ under isomorphism $\bar{\vp} =\eta\circ \vp \circ \eta ^{-1}$.  This isomorphism is contained in the endomorphism $\bar{\sigma} =\eta\circ \sigma \circ \eta ^{-1}$ and the inverse isomorphism $\bar{\vp} ^{-1}$ is contained in the endomorphism $\bar{\tau} =\eta\circ \tau \circ \eta ^{-1}$. Thus we can apply the  fact proved above in the point {\bf 1}, that is, $\bar{\vp}$ can be extended up to automorphism $\bar {\Phi}$ of $G$.

Consider this extension in details. A base $g_1,...,g_n$ of Abelian group $G$ and a base $\bar {a}_1,...,\bar{a}_k$ of its subgroup $\bar{A}$ are chosen such that $\bar{a}_i = g_i ^{p_i}$ for $1\leq i \leq k$ (now we use the multiplicative notation).
The automorphism $\bar {\Phi }$ of $G$ extending $\bar{\vp}$ is constructed  in  such a way that $\bar{\Phi}(g_i)=\bar{\sigma} (g_i)$ for  $i\leq k$. The elements  $f_i =\bar{\Phi}(g_i)$ for $i=1, \dots , n$ form a base of $G$ in which first $k$ elements are equal to corresponding $\bar{\sigma} (g_i)$. 

It is known from the theory of nilpotent groups (see for example \cite{Kurosh}) that  a system $h_1, \dots ,h_n$ of elements of  $H$ is a system of free generators of some free nilpotent subgroup of the same class if and only if the the system $\eta (h_1), \dots,\eta (h_n) $ is linear independent in $G=H/H'$. So if $\eta (h_1), \dots,\eta (h_n) $ is a base of $G$ then $h_1, \dots ,h_n$ is a base of a free nilpotent subgroup $H_0$ of $G$. Since $\eta (H_0)=G$, we have $H_0H'=H$. The last one implies that $H_0=H$. We obtain that if $\eta (h_1), \dots,\eta (h_n) $ is a base of $G$ then $h_1, \dots ,h_n$ is a base of $H$. Below we apply this property of  finitely generated free nilpotent groups.

There exist bases $h_1, \dots ,h_n$ and $u_1, \dots ,u_n$ of $H$ such that $\eta (h_i)=g_i$ and $\eta (u_i)= f_i$ for $1\leq i \leq n$. Of course we can chose $u_i =\sigma (h_i)$ for  $1\leq i \leq k$ because $\eta (\sigma (h_i))=\bar{\sigma} (g_i)=f_i$ for  $1\leq i \leq k$. 

Now we define an automorphism $\Phi$ of $H$ setting $\Phi (h_i)=u_i$ for $1\leq i \leq n$. On the other hand, elements $h_i^{p_i}$ $1\leq i \leq k$ form a base of the free nilpotent subgroup $AH'$ because $\eta (h_i^{p_i})=g_i^{p_i}=\bar{a_i}$. We have $\Phi (h_i^{p_i})=(\Phi (h_i))^{p_i}=u_i^{p_i}=(\sigma (h_i))^{p_i}=\sigma (h_i ^{p_i})$. Thus $\Phi$ coincides with $\sigma$ on the subgroup $AH'$. Since $\sigma$ contains $\vp$ which is defined on $A\subset AH' $, $\Phi$ is an extension of $\vp$. 
\end{proof}

\begin{lem}\label{Sem}
Every finitely generated free semigroup is weakly homogeneous. 
\end{lem}
\begin{proof}
Let $S$ be a free semigroup with the set $X=\{x_1,\dots, ,x_k \}$ of free generators . Let $\vp :A\to B$ be an automorphism between two subsemigroups $A$ and $B$ of $S$, where $A$ and $B$ are generated by elements  $a_1, \dots , a_n$ and $b_1, \dots , b_n$ respectively. We may assume that $\vp (a_i)=b_i$ for $1\leq i\leq n$. 

Suppose that there exist two endomorphisms $\sigma$ and $\tau$ first of which extends $\vp$ and the second one extends $\vp ^{-1}$. Thus $\sigma (a_i)=b_i$ and $\tau (b_i)=a_i$  Denote by $\vert w\vert$ the length of the word $w$ in alphabet $X$. Since  $\vert \sigma (w)\vert \geq\vert w\vert$ and $\vert \tau (w)\vert \geq\vert w\vert$ for every $w\in S$, we obtain that $\vert a_i\vert =\vert b_i\vert $. Let $y_1, \dots ,y_p$ be the list of all variables from $X$ which occur in $a_1, \dots , a_n$ and 
$z_1, \dots ,z_q$  be the analogical  list of all variables which occur in $b_1, \dots , b_n$. It is obvious that  $\vert \sigma (y_i)\vert =1$ for all $1\leq i\leq p$ and $\vert \tau (z_i) \vert =1$ for all $1\leq i\leq q$.  Therefore we have that $\sigma (y_i) \in \{z_1, \dots ,z_q \}$ and $\tau (z_i) \in \{y_1, \dots ,y_p \}$. 

Since $\tau (\sigma (a_i))=a_i$ and $\sigma (\tau (b_i))=b_i$ for  $1\leq i\leq n$, we have that the restrictions of $\sigma $ and $\tau$ to variables $y_1, \dots ,y_p $ and $z_1, \dots ,z_q $ respectively are mutually inverse maps. Thus $p=q$ and $\sigma $ and $\tau$ induce two  mutually inverse partial one-to-one transformations of $X$. Let $\alpha $ be a bijection of $X\setminus \{y_1, \dots ,y_p\} $ on $X\setminus \{z_1, \dots ,z_p\}$.  Setting $\tilde {\vp } (y_i)=\sigma (y_i)$   for $1\leq i\leq p$ and $\tilde{\vp } (x)=\alpha (x)$ for all other variables from $X$, we obtain the automorphism  $\tilde{\vp}$ of $S$ which extends $\vp$. 
\end{proof}
Lemmas \ref{abel}, \ref{Sem} and \ref{W_Hom} give us the following result:
\begin{theorem}\label{three log.perf}
 Finitely generated free Abelian groups, finitely generated free nilpotent groups of any class and finitely generated semigroups are logically perfect.
\end{theorem}
The method which has been used to prove the theorem above can not be applied  to non-Abelian finitely generated  free groups. 
\begin{prop}\label{contrexam}
Free groups of rank 2 are not weakly homogeneous.
\end{prop}
\begin{proof}
Consider the free group $\mathbb{F}_2$ of rank 2 free generated by $x_1,x_2$. Let $a=x_1^2 x_2 x_1^{-1}x_2$ and $b=x_1x_2$. Define endomorphisms $\sigma$ and $\tau$ setting $\sigma (x_1)=x_1x_2,\; \sigma (x_2)=1$ and $\tau (x_1)=x_1 ^2x_2,\; \tau (x_2)=x_1 ^{-1}x_2$. We see that $\sigma (a)=b$ and $\tau (b)=a$. Thus $\sigma $ induces an isomorphism of $\vp :\langle a \rangle\to \langle b \rangle$ and $ \tau$ induces the inverse isomorphism $\vp ^{-1}$.

Suppose that there exists an automorphism $\tilde{\vp }$ of  $\mathbb{F}_2$ which sends $a$ to $b$. Let $\tilde{\vp }(x_1) =w_1$ , $\tilde{\vp } (x_2) =w_2$, where $w_1$,$w_2$ are  words in symbols $x_1,x_2$.
Thus we have a relation in our free group: $x_1x_2 \equiv w_1^2 w_2 w_1^{-1}w_2$. (*)
  
This relation must be an identity in the group variety. 
Let $l_1, l_2$ be the sums of all exponents of $x_1, x_2$ incoming in $w_1$ and  $m_1,m_2$  the sums of all exponents of $x_1, x_2$ incoming in $w_2$ respectively. It is obvious that $l_1+2m_1=l_2+2m_2=1$. Thus $l_1, l_2$ must be odd numbers.

Consider the group $S_3$ of all permutations of the set $\{1,2,3\}$ . This group is a homomorphic image of $\mathbb{F}_2$ under the map $\gamma$ which maps  $x_1$ to $(213)$ and $x_2$ to $(132)$. Since $\gamma (x_1^2)=\gamma (x_2^2 ) =(123), \; \gamma (x_1x_2)=(312), \; \gamma (x_2x_1)=(231),\; \gamma (x_1x_2x_1) =\gamma (x_2x_1x_2)=(321), \; \gamma ((x_1x_2)^2)=\gamma (x_2x_1),\; \gamma ((x_2x_1)^2)=\gamma (x_1x_2)$, we obtain that  the following equalities are satisfied in $S_3$:  $w_1\equiv x_1x_2$ or $w_1\equiv x_2x_1$. For $w_2$ we have variants:  $ w_2\equiv 1, x_1, x_2,x_1x_2, x_2x_1, x_1x_2x_1$. Since $w_1,w_2$   generate $\mathbb{F}_2$, their images generate $S_3$. Therefore we have only three variants for $w_2$: $w_2\equiv x_1, x_2, x_1x_2x_1$.  Directly calculations show that in all mentioned cases  $\gamma (w_1^2 w_2 w_1^{-1}w_2)=(123)$ which contradicts to the identity (*).

Consequently  there is no automorphism of $\mathbb{F}_2$ sending $a$ to $b$. 
\end{proof}
Nevertheless all free finitely generated non-Abelian  free groups are logically perfect. This fact is proved in \cite{chloe} in view of Theorem \ref{twoTypes}.
 

\section{Isotyped algebras}\label{Iso}
We consider the following problem: in what cases isotyped algebras are necessarily isomorphic. At first, we generalize the result obtained in \cite{PZ}, Theorem 3.11.
\begin{theorem}\label{NecessaryIso}
If two algebras $H_1$ and $H_2$ from the same variety $\Theta$ are isotyped then for every finitely generated subalgebra $A$  of $H_1$ there exists a subalgebra $B$ of  $H_2$ isomorphic to $A$, and if $A$ is a proper subalgebra then $B$ can be chosen as a proper subalgebra too.
\end{theorem}
\begin{proof}
Let $H_1$ and $H_2$ be isotyped $\Theta$-algebras. Let $A=\langle a_1, \dots ,a_n\rangle$ where $a_1, \dots ,a_n$ are different elements in $H_1$.  Consider  the free $\Theta$-algebra $W(X)$, where $X=\{x_1, \dots ,x_n \}$. Let  $\nu \in \hom (W(X),H_1)$ defined by $\nu (x_i)=a_i$ for $1\leqq i \leqq n$. Since  $H_1$ and $H_2$ are isotyped there exists a point $\mu \in \hom (W(X),H_2)$ such that $LKer\nu = LKer \mu$. 
 We obtain a subalgebra $B=\langle \mu (a_1) ,\dots ,\mu (a_n)\rangle$ of $H_2$ and $B=\mu (W(X))$. Since $Ker\nu =Ker \mu$, algebras $A$ and $B$ are isomorphic. 

Let now $A$ be a proper subalgebra of $H_1$ and let $a_{n+1} \in H_1 \setminus A $. Add  to $X$ a new variable $x_{n+1}\not \in X$ and consider a new point $\nu :W(X\cup \{x_{n+1}\})\to H_1$ setting $\nu (x_i) =a_i$ for all $1\leqq i \leqq n+1$.  
For every $w\in W(X)$ consider the following formula $v_w \in \Phi(X\cup  \{x_{n+1}\})$:
$$v_w=\neg (x_{n+1}\equiv w) .$$
Under condition that $H_1$ and $H_2$ are isotyped, there exists a point $\mu \in \hom (W(X\cup \{x_{n+1}\}),H_2)$  such that
$LKer\nu = LKer \mu$. Since $LKer\nu \cap M_X = LKer \mu \cap M_X$, the subalgebra $B$  generated by $\mu (x_1), \dots , \mu (x_n)$ is isomorphic to $A$. On the other hand,  it is obvious that $v_w\in LKer \nu$  and hence $v_w\in LKer \mu$  for every $w\in W(X)$. The last one means that $\mu (x_{n+1})$ does not belong to $B$, that is, $B$ is a proper subalgebra of $H_2$.

\end{proof}
\begin{cor} Let a finitely generated algebra $H$  contain no proper  subalgebra isomorphic to $H$. Then  every algebra $G$ isotyped  to  $H$  is isomorphic to $H$.
\end{cor}
\begin{proof}
Let $H$ and $G$ be isotyped algebras. Since $H$ is finitely generated, there exists a subalgebra $B$ of $G$ isomorphic to $H$. If $B$ is a proper subalgebra of $G$ then $H$ contains a proper subalgebra $A$ which is  isomorphic to $B$ and therefore $A$ is  isomorphic to $H$ but this is impossible according to the hypotheses. Thus $B=G$.
\end {proof}
We can  apply this result to finitely dimensional linear spaces  but it is not the case for finitely generated free Abelian groups. However the next result can be obtained using Theorem \ref{NecessaryIso}.
\begin{theorem}\label{IsoAbel} If two Abelian groups are isotyped and one of them is  free and finitely generated then they are isomorphic.
\end{theorem}
\begin{proof}
Let $H$ and $G$ be  isotyped  Abelian groups and  $H$ be free of rank $n$. Then every finitely generated subgroup of $G$ is isomorphic to a subgroup of $H$.  Therefore every finitely generated subgroup of $G$ is free of a  rank $k\leq n$. This means that  every $n+1$ elements of $G$ are linearly dependent. On the other hand, $H$ is isomorphic to a subgroup $B$ of $G$.  Let $g_1, \dots, g_n$ is a base of  $B$. These elements form a maximal linearly independent system in $G$. We obtain that rank of  $G$ is equal to $n$. 

It remains to show that $G$ is finitely generated. Let $h_1,\dots , h_n$ be a base of $H$. Consider the following countable set of formulas  $u_{(q_1,\dots ,q_n)}(x_1, \dots ,x_n)$, indexed by $n$-tuples $(q_1,\dots ,q_n)$ of integers, which not all are equal to zero and formulas $v_{(q_1,\dots ,q_n,q)}(x_1, \dots ,x_n)$, indexed by $n+1$-tuples $(q_1,\dots ,q_n,q)$ of integers , where $q\not =0$ :
$$
u_{(q_1,\dots ,q_n)}(x_1, \dots ,x_n)= q_1x_1+q_2x_2+\dots +q_n x_n \not \equiv 0,
$$
\begin{multline*}
 v_{(q_1,\dots ,q_n,q)}(x_1, \dots ,x_n)=\forall y (q_1x_1+q_2x_2+\dots +q_n x_n+q y\equiv 0  \\
\Longrightarrow \bigvee_{\vert k_i \vert \leq \vert \frac{q_i}{q}\vert , i=1,\dots n} y\equiv k_1x_1+\dots +k_n x_n).
\end{multline*}

Every such formula is satisfied in $H$ by the tuple $\bar{h}=(h_1,\dots , h_n)$. Indeed, for the formulas $u_{(q_1,\dots ,q_n)}(x_1, \dots ,x_n)$ this statement is obvious. Consider the formulas $v_{(q_1,\dots ,q_n,q)}$. 
Suppose that for an element $h\in H$ we have $q_1 h_1+q_2 h_2+\dots +q_n h_n+q h=0$ for some integers $(q_1,\dots ,q_n,q)$ and $q\not =0$. Since  $(h_1,\dots , h_n)$ is a base, $h=k_1h_1+\dots +k_n h_n$ for some integers $k_i, \; i=1,\dots ,n$. It obvious that $k_i=-\frac{q_i}{q}$. Thus all considered formulas belong to $tp^{H}(\bar{h})$. 

Since $H$ and $G$ are isotyped, all formulas $u_{(q_1,\dots ,q_n)}$ and $v_{(q_1,\dots ,q_n,q)}$  belong to $tp^{G}(\bar{g})$ for some $n$-tuple $\bar{g}=(g_1, \dots ,g_n)$ in $G$. First of all this means that elements $g_1, \dots ,g_n$ are linearly independent. Let $g$ be an arbitrary element in $G$. Since rank of $G$ is $n$, the elements $g_1, \dots ,g_n,g$ are linearly dependent, that is,  $q_1g_1+\dots +q_n g_n+q g=0$ for some integers $(q_1,\dots ,q_n,q)$, which not all are equal to zero. Taking  into account that  the first $n$ elements are linearly independent, we conclude that $q\not =0$. Since   $v_{(q_1,\dots ,q_n,q)}(g_1, \dots ,g_n)$ is valid in $G$, we obtain that 
$$ \bigvee_{\vert k_i \vert \leq \vert \frac{q_i}{q}\vert, i=1,\dots n} g= k_1g_1+\dots +k_n g_n.$$  
This means that $g=k_1g_1+\dots +k_n g_n$ for some integers $k_1,\dots ,k_n$.

Consequently $G$  it is generated by $g_1, \dots, g_n$, and therefore $G$ is  isomorphic to $H$.
\end{proof}

{\bf Conjecture.} It seems to be probable that analogous result takes place for nilpotent groups too.

{\bf Remark} B. Plotkin writes \cite{Plotkin_IsotAlg} that Z. Sela has proved a similar fact for free non-commutative groups (unpublished).

\end{document}